\documentclass[12pt,a4paper]{article}
\pdfoutput=1 
\usepackage[T1]{fontenc}
\usepackage[utf8]{inputenc}
\usepackage[english]{babel}
\usepackage{authblk}
\usepackage{amsmath,amssymb,amsthm}
\usepackage{graphicx}
\usepackage{color}
\usepackage{hyperref}
\usepackage{todonotes}
\usepackage{amsfonts}
\usepackage{mathrsfs}
\usepackage{url}
\newtheorem{theorem}{Theorem}
\newtheorem{lemma}{Lemma}
\newtheorem{remark}{Remark}

\allowdisplaybreaks

\begin{document}

\begin{center}
	\textbf{\Large {Stability and instability of steady states for a branching
random walk}}
\end{center}

\begin{center}
	\textit{{\large Yaqin Feng, Stanislav Molchanov, Elena Yarovaya}}
\end{center}

\begin{abstract}
We consider the time evolution of a lattice branching random walk
with local perturbations. Under certain conditions, we prove the Carleman
type estimation for the moments of a particle subpopulation number  and show the existence of a steady
state.

\medskip

\end{abstract}

\section{Introduction}

There are several models in population dynamics, in which steady
states (statistical equilibrium) exist.
Typical models in this area are based on the concept of branching random walk.
Numerous
variations and versions of  branching random walk models have been
studied and used to describe the population dynamics, see \cite{1998,Bogachev_1998a,Bogachev_1998b,Feng_2012,Koralov_2013,SM_JW_2017,Molchanov_Yarovaya_2012,Molchanov_Yarovaya_2012_a,Molchanov_Yarovaya_2012_b,Molchanov_Yarovaya_2012_c,Yarovaya_2007} and bibliography therein.

The simplest model of such a type, which we call\emph{ the contact critical model},
in the lattice case has the following structure.
We consider a random field $n(t,\cdot)$ of  particles on $\mathbb{Z}^{d}$, $d\ge 1$, where $n(t,y)$ is the number of particles at the point $y \in \mathbb{Z}^{d}$ at the time moment $t\ge 0$.
At the moment $t=0$, there is a single particle at each point $x\in \mathbb{Z}^{d}$,
that is $n(0,x)\equiv1$. Each of such initial particles, independently of others,
generates its own subpopulation of the particle offsprings $n(t,y,x)$ at the point $y\in \mathbb{Z}^{d}$ at the moment
$t$. Then the total population of the particle offsprings over $\mathbb{Z}^{d}$ is defined by $n(t,y):=\sum_{x\in \mathbb{Z}^{d}}n(t,y,x)$.

The evolution of the subpopulation $n(t,y,\cdot)$ includes the random
walk of the particles until the first transformation. The underlying random
walk is described by the generator
\begin{eqnarray*}
\mathit{\mathcal{L}_{a}}f(x)=\sum_{z\in \mathbb{Z}^{d}}(f(x+z)-f(z))a(z)
\end{eqnarray*}
with $a(z)=a(-z)$, $\sum_{z\neq0}a(z)=1$ and $\mathop{\mathrm{span}}\{z_{j}:\,a(z_{j})>0\}=\mathbb{Z}^{d}$.
The transformation includes the splitting (i.e. the splitting of each
particles into two particles at the same lattice point as parental one) with
the rate $\beta>0$ and the annihilation with the mortality rate $\mu$. The
central assumption of criticality for such branching process at every lattice point  is $\beta=\mu$.
\begin{remark}\label{R1}
A related result for the continuum $\mathbb{R}^{d}$ was obtained under some additional conditions in \cite{Kondratiev_2008,Kondratiev_2016}. It is based on the forward Kolmogorov equations for the correlation functions and is not applicable for the branching random walks on $\mathbb{Z}^{d}$.
\end{remark}
The following theorem gives the final condition for the existence
of the steady state, see, e.g. \cite{SM_JW_2017}.
\begin{theorem}\label{T1}
If, for a BRW under consideration, the underlying random walk with the generator $\mathcal{L}_{a}$
is transient, that is
\[
\int_{[-\pi,\pi]^{d}}\frac{1}{1-\hat{a}(k)}dk<\infty,
\]
 where $\hat{a}(k)=\sum_{z\neq0}\cos (k,z)a(z)$, then
\[
n(t,y)\xrightarrow{\text{Law}}n(\infty,y),
\]
where $n(\infty,y)$ is a steady state. If the random walk is recurrent,
then
\[
n(t,y)\xrightarrow{\text{Law}}0.
\]
\end{theorem}

\begin{remark}\label{R2}
Note that for a recurrent random walk, the assertion of Theorem~\ref{T1} corresponds to the phenomenon of
clusterization: for large $t$, the population consists of ``large
islands'' of particles  the distance between which increases with increasing $t$.
\end{remark}

Proof of Theorem~\ref{T1}
is based on the analysis of the moments and correction functions.
First, it is necessary to prove that, for any positive integer $l$,
the moment $En^{l}(t,y)\xrightarrow{\text{}}M_{l}$, as $t\to \infty$, and then check
the Carleman condition $M_{l}\leq l!C^{l}$ for some constant $C$.
The central point in this proof is the observation that $n(t,y)=\sum_{x\in \mathbb{Z}^{d}}n(t,x,y)$,
where the subpopulations $n(t,x,y)$ for different $x\in \mathbb{Z}^{d}$ are independent.
Instead of moments, we will use cumulants and the following properties
of the cumulants: the $k$-th cumulant of a sum of independent random
variables is just the sum of the $k$-th cumulants of the summands. For
the calculation of the $k$-th cumulant of $n(t,x,y)$, we can use the
backward Kolmogorov equations.

In general, even for non-constant $\beta(x)$ and $\mu(x)$, let $V(x)=\beta(x)-\mu(x)$,
all moment equations include the basic operator
\[
\mathit{\mathcal{H}_{a}}=\mathit{\mathcal{L}_{a}+V(x).}
\]
In fact, it is easy to see that the moment generation function of $n(t,x,y)$,
which is $\phi_{z}(t,x,y):=E_{x}z^{n(t,x,y)}$, is the solution of the KPP type non-linear equation:
\begin{eqnarray*}
\frac{\partial\phi_{z}}{\partial t} & = & \mathcal{L}_{a}\phi_{z}+\beta(x)\phi_{z}^{2}-(\beta(x)+\mu(x))\phi_{z}+\mu(x),\\
\phi_{z}(0,x,y) & = & \begin{cases}
z, & x=y,\\
1, & x\neq y.
\end{cases}
\end{eqnarray*}
Denote $m_{l}(t,x,y):=\frac{\partial^{l}\phi_{z}}{\partial z^{l}}|_{z=1}$,
then the first moment satisfies the following equation:
\begin{align*}
\frac{\partial m_{1}}{\partial t} & =  \mathcal{H}_{a}m_{1},\\
m_{1}(0,x,y) & =  \delta_{x}(y).
\end{align*}

In a similar way, for all $l\geq2$, we can find the equation for
the $l$-th factorial moment $m_{l}(t,x,y):=E\left(n(t,x,y)\left(n(t,x,y)-1\right)\cdots\left(n(t,x,y)-l+1\right)\right)$
\begin{align*}
\frac{\partial m_{l}}{\partial t} & = \mathcal{H}_{a}m_{l}+2\beta(x)\sum_{i=1}^{l-1}\left(\begin{array}{c}
l-1\\
i
\end{array}\right)m_{i}m_{l-i}\\
m_{l}(0,x,y) & =  0.
\end{align*}
The first moment equation gives
\[
m_{1}(t,x,y)=q(t,x,y),
\]
where $q(t,x,y)$ is the fundamental solution of the equation
\[
\frac{\partial q(t,x,y)}{\partial t}=\mathcal{H}_{a}q(t,x,y).
\]

If $\beta=\mu$, then $V(x)\equiv0$ and $m_{1}(t,y)\equiv1$. The
density of the global population is constant in time. One can solve
one by one the moment equations and finally prove Theorem 1. For more
details, please refer to \cite{SM_JW_2017}.

But the assumption of criticality $\beta=\mu$ is not realistic. Any
model in population dynamics measured, at least on the qualitative
level, on the similarity with nature bio-systems must be stable with
respect to small random perturbation of the form $\beta(x)=\beta_{0}+\varepsilon\xi(x,w)$,
$\mu(x)=\mu_{0}+\varepsilon\eta(x,w)$, with $\beta_{0}=\mu_{0}$ and
$\omega\in(\varOmega,\mathcal{F},P)$. Consider the simplest case
(a stationary Anderson type model) where $V(x,\omega)=\varepsilon(\beta-\mu)(x,\omega)$
are i.i.d random variable and $P\{V(x)>\delta_{0}\}=\delta_{1}$ for
$\delta_{0}>0$, $\delta_{1}>0$. Due to the Borel\textendash Cantelli
lemma, for arbitrary large $L>0$, there is a cube $Q_{L}(x)=\{x:|a-x|_{\infty}\leq L\}$
such that $V(x)>\delta_{0}$ for $x$$\in Q_{L}(x)$. Then
\[
m_{1}(t,x)\geq\tilde{m}_{1,L}(t,x),
\]
 where
\begin{align*}
\frac{\partial\tilde{m}_{1,L}}{\partial t} & = \mathcal{L}_{a}\tilde{m}_{1,L}+\delta_{0}\tilde{m}_{1,L}=\mathcal{H}_{0}\tilde{m}_{1,L},\\
\tilde{m}_{1,L}|_{\partial Q_{L}} & =  0, \\
\tilde{m}_{1,L}(0,x) & = \psi_{0}(x)=\cos r(x_{1}-a_{1})\cdots \cos r(x_{d}-a_{d}).
\end{align*}
 Here $a=(a_{1},\ldots,a_{d})$, $x=(x_{1},\ldots,x_{d})$ and $r=\frac{\pi}{2L}$. The function
$\psi_{0}(x)=\cos r(x_{1}-a_{1})\cdots \cos r(x_{d}-a_{d})$ is the eigenfunction
of $\mathcal{H}_{0}$ with the eigenvalue $\lambda_{0}\sim-\frac{c}{L^{2}}+\delta_{0}$.
For sufficiently large $L$, this eigenvalue is positive, i.e. $m_{1}(t,x)=\psi_{0}(x)e^{\lambda_{0}t}$
is exponentially growing in time. This discussion indicates that the contact model is unstable with respect to small stationary random perturbations. See details in~\cite{Kondratiev_Molchanov_2016}.

\section{The case of a single-source perturbation}

Let us now consider the local perturbations and concentrate, in this section, on the simplest case:
\[
\beta(x)=\beta_{0}+\sigma I_{\{x=0\} },\quad \sigma>0
\]
 and
\[
\mu(x)=\beta_{0}
\]
 for all $x\in \mathbb{Z}^{d}$, i.e. $V(x)=\sigma\delta_{0}(x)$.

\subsection{Transition probability and the Green function for the operator \textmd{\normalsize{}$\mathcal{L}_{a}$}}
\begin{lemma}
The fundamental solution of
\[
\begin{cases}
\frac{\partial p(t,x,y)}{\partial t} & =\mathcal{L}_{a}p(t,x,y),\\
p(0,x,y) & =\delta_{x}(y),
\end{cases}
\]
is
\begin{eqnarray*}
p(t,x,y) & =\frac{1}{\left(2\pi\right)^{d}}\int_{T^{d}} & e^{t\mathcal{\hat{L}}_{a}(k)}e^{-ik(y-x)}dk,
\end{eqnarray*}
where $k\in T^{d}:=[-\pi,\pi]^{d}$, $\mathcal{\hat{L}}_{a}(k)=-(1-\hat{a}(k))$ and $\hat{a}(k)=\sum_{z\neq0}\cos (k,z)a(z)$.
\end{lemma}


Note that
\[
p(t,x,y)\leq p(t,x,x)=p(t,0,0) =\frac{1}{\left(2\pi\right)^{d}}\int_{T^{d}} e^{t\hat{L}_{a}(k)}dk.
\]

Let us consider, for  the operator $\mathcal{H}_{a}:=\mathcal{L}_{a}+\sigma\delta_{0}(x)$,  the
following equation
\[
\mathcal{H}_{a}\psi=\lambda\psi,  \psi\in L^{2}\left(\mathbb{Z}^{d}\right),\quad\lambda>0,
\]
that is
\[
\mathcal{L}_{a}\psi+\sigma\delta_{0}(x)\psi=\lambda\psi,
\]
The Fourier transform gives
\[
-(1-\hat{a}(k))\hat{\psi}(k)+\sigma\psi(0)=\lambda\hat{\psi}(k),
\]
 therefore,
\[
\frac{1}{(2\pi)^{d}}\int_{T^{d}}\hat{\psi}(k)dk=\frac{1}{(2\pi)^{d}}\int_{T^{d}}\frac{\sigma\psi(0)}{\lambda+(1-\hat{a}(k)}dk.
\]
Let $I(\lambda)$ be the function
\[
I(\lambda):=\frac{1}{(2\pi)^{d}}\int_{T^{d}}\frac{1}{\lambda+(1-\hat{a}(k)}dk.
\]
Using the fact that $\psi(0)=\frac{1}{(2\pi)^{d}}\int_{T^{d}}\hat{\psi}(k)dk$
we have
\[
\frac{1}{\sigma}=I(\lambda).
\]

Let $G_{\lambda}(x,y)$ denote the Green function of the operator $\mathcal{L}_{a}$,
which is defined as the Laplace transform of the transition probability $p(t,x,y)$:
\[
G_{\lambda}(x,y):=\int_{0}^{\infty}e^{-\lambda t}p(t,x,y)\,dt
\]
for $\lambda\geq0$. We have
\[
G_{\lambda}(x,y) =\int_{0}^{\infty}e^{-\lambda t}p(t,x,y)\,dt=\frac{1}{(2\pi)^{d}}\int_{T^{d}}\frac{e^{-ik(y-x)}}{\lambda+1-\hat{a}(k)}\,dk
\]
and
\[
G_{0}(0,0) =\int_{0}^{\infty}p(t,0,0)\,dt=\frac{1}{(2\pi)^{d}}\int_{T^{d}}\frac{1}{1-\hat{a}(k)}\,dk=I(0).
\]
For the random walk, $G_{0}(0,0)$ denotes the total time that the random walk stays in the origin if it start from 0. If $\sum_{z\in \mathbb{Z}^{d}}a(z)|z|^2 < \infty$, then $G_{0}(0,0) < \infty$ for $d\geq 3$. Please refer to \cite{Yarovaya_2007} for more discussion of $G_{0}(0,0)$.

In the following, we  assume that  $\sum_{z\in \mathbb{Z}^{d}}a(z)|z|^2 < \infty$. It means that the underlying random walk has  a finite variance of jumps.
The asymptotic of the Green function is studied in \cite{Molchanov_Yarovaya_2012,Molchanov_2013}. We recall the following lemma from \cite{Molchanov_Yarovaya_2012}.
\begin{lemma}
\label{lem:Green_function_asym}Suppose that $d\geq3$, assume $\sum_{z\in \mathbb{Z}^{d}}a(z)|z|^2 < \infty$,  then
\begin{eqnarray*}
G_{0}(x,y)\sim\frac{C_{d}}{|y-x|^{d-2}},\quad\text{as}~ |y-x|\rightarrow\infty,
\end{eqnarray*}
where $C_{d}$ is a positive constant.
\end{lemma}

\subsection{The First moment }

From the previous section, for the first
moment $m_{1}(t,x,y)=En(t,x,y)$,  we have the following equation
\begin{align*}
\frac{\partial m_{1}(t,x,y)}{\partial t} & = \mathcal{L}_{a}m_{1}(t,x,y)+\sigma\delta_{0}(x)m_{1}(t,x,y)\\
m_{1}(0,x,y) & =\delta_{x}(y).
\end{align*}
From the Kac-Feyman formula, we have the solution for $m_{1}(t,x,y)$:
\[
m_{1}(t,x,y)  =p(t,x,y)E_{x}\left[e^{\sigma\int_{0}^{t}\delta_{0}(X_{s})ds}|X_{t}=y\right]
\]
Here $X_{t}$ is the underlying random walk with the generator $\mathcal{L}_{a}$
and $E_{x}$ is the expectation under the condition that the random
walk starts at the point $x$.

Let $m_{1}(t,y)=En(t,y)=\sum_{x\in \mathbb{Z}^{d}}m_{1}(t,x,y)$, then
\begin{align*}
\frac{\partial m_{1}(t,y)}{\partial t} & = \mathcal{L}_{a}m_{1}(t,y)+\sigma\delta_{0}(x)m_{1}(t,y),\\
m_{1}(0,y) & =  1.
\end{align*}
Then
\[
m_{1}(t,0)  =E_0\left[e^{\sigma\int_{0}^{t}\delta_{0}(X_{s})ds}\right]
\]
As $t\rightarrow\infty$,  we  find $E_0\left[e^{\sigma\int_{0}^{\infty}\delta_{0}(X_{s})ds}\right]$ in the following explicit form:
\begin{equation}\label{equation: formulam_1_infinity}
E_0\left[e^{\sigma\int_{0}^{\infty}\delta_{0}(X_{s})ds}\right]=\frac{1}{1-\sigma G_{0}(0,0)}.
\end{equation} 
From  the  last formula we obtain the following lemma.
\begin{lemma}
\label{lem:Lemma eigenvalue} Suppose that $d\geq3$,  then $E_0\left[e^{\sigma\int_{0}^{\infty}\delta_{0}(X_{s})ds}\right]<\infty$ if and
only if $\sigma< G^{-1}_0(0,0)$.
\end{lemma}
\begin{theorem}
If  $d\ge3$ and $\sigma< G^{-1}_{0}(0,0)$  then we obtain
\[
m_{1}(t,0)\xrightarrow[t\rightarrow\infty]{}m_{1}(\infty,0)=E_{0}\left[e^{\sigma\int_{0}^{\infty}\delta_{0}(X_{s})\,ds}\right].
\]
\end{theorem}
%

From the Kac-Feyman formula, we have the solution for $m_{1}(t,x,y)$:
\[
m_{1}(t,x,y)  =p(t,x,y)E_{x}\left[e^{\sigma\int_{0}^{t}\delta_{0}(X_{s})ds}|X_{t}=y\right].
\]
First we  prove the following lemma.
\begin{lemma}\label{L4}\sloppy
The following inequalities are valid:  $m_{1}(t,x,0)\leq m_{1}(t,0,0)$ and $m_{1}(t,x,y)\leq m_{1}(t,0,0)$.
\end{lemma}

\begin{proof}
Denote
\[
\tau_{x,0}:=inf\,\{t\,|X_{t}=0|X_{0}=x\}.
\]
Then
\begin{eqnarray*}
m_{1}(t,x,0) & = & p(t,x,0)E_{x}\left[e^{\sigma\int_{0}^{t}\delta_{0}(X_{s})ds}|X_{t}=0\right]\\
 & = & p(t,x,0)E_{x}\left[e^{\sigma\int_{0}^{\tau_{x,0}}\delta_{0}(X_{s})ds}e^{\sigma\int_{\tau_{x,0}}^{t}\delta_{0}(X_{s})ds}|X_{t}=0\right]\\
 & = & p(t,x,0)E_{0}\left[e^{\sigma\int_{\tau_{x,0}}^{t-\tau_{x,0}}\delta_{0}(X_{s})ds}|X_{t}=0\right]\\
 & = & p(t,x,0)E_{0}\left[e^{\sigma\int_{0}^{t-\tau_{x,0}}\delta_{0}(X_{s})ds}|X_{t}=0\right]\\
 & \leq & p(t,0,0)E_{0}\left[e^{\sigma\int_{0}^{t}\delta_{0}(X_{s})ds}|X_{t}=0\right]\\
 & = & m_{1}(t,0,0).
\end{eqnarray*}
Here the proof of the second equality uses the fact that before the moment $\tau_{0,x}$, the random
walk does not reach zero, thus $\int_{0}^{t}\delta_{0}(X_{s})ds=0$. The
third equality is based on the Markov property of the random walk and the
second last equality uses the definition of $m(t,x,y)$.
In a similar way, we can prove that
\begin{align*}
m_{1}(t,x,y)\leq p(t,x,y)E_{0}\left[e^{\sigma\int_{0}^{t}\delta_{0}(X_{s})ds}.\right]
\end{align*}
Denote $A_{t}:=E_{0}\left[e^{\sigma\int_{0}^{t}\delta_{0}(X_{s})ds}\right]$,
that is
\[
m_{1}(t,x,y)\leq p(t,x,y)A_{t}
\]
Due to the fact that $p(t,x,y)\leq p(t,0,0)$, we have
\[
 m_{1}(t,x,y)\leq m_{1}(t,0,0).
\]
\end{proof}

Note that $m_{1}(t,x,y)= m_{1}(t,y,x)$. The estimation $m_{1}(t,x,y)\leq m_{1}(t,0,0)$ gives the possibility to extend the method from \cite{SM_JW_2017} to a local perturbation situation.

\subsection{Higher moments}

As we know, for $l\ge2$, the $l$-th factorial moment satisfies:
\begin{align*}
\frac{\partial m_{l}}{\partial t}&=  \mathcal{L}_{a}m_{l}+\sigma\delta_{0}(x)m_{l}+2(\mu+\sigma\delta_{0}(x))\sum_{i=1}^{l-1}
\left(\begin{array}{c}
l-1\\
i
\end{array}\right)m_{i}m_{l-i}\\
m_{l}(0,x,y)&=0.
\end{align*}

We will prove the estimates given below for the factorial moment.
\begin{theorem}
\label{thm: single_source_moment}For $d\ge3$, we have
\[
m_{l}(t,x,y)\leq A^{l-1}B^{l}l!p(t,x,y),
\]
where $A:=E_{0}\left[e^{\sigma\int_{0}^{\infty}\delta_{0}(X_{s})ds}\right]$
and $B:=2(\mu+\sigma)G_{0}(0,0)$.
\end{theorem}

\begin{proof}
We will use the mathematical induction to prove the theorem. For $l=2$,
the second moment satisfies
\begin{align*}
\frac{\partial m_{2}}{\partial t} & = \mathcal{L}_{a}m_{2}+\sigma\delta_{0}(x)m_{2}+2(\mu+\sigma\delta_{0}(x))m_{1}^{2}\\
m_{2}(0,x,y) & = 0.
\end{align*}
From Duhamel's formula, we have
\[
\begin{split}m_{2}(t,x,y)& = 2(\mu+\sigma\delta_{0}(x))\int_{0}^{t}\sum_{v\in\mathbb{Z}^{d}}m_{1}(t-s,x,v)m_{1}^{2}(s,v,y)\,ds\\
& \leq  2(\mu+\sigma)\int_{0}^{t}\sum_{v\in\mathbb{Z}^{d}}p(t-s,x,v)A_{t-s}p^{2}(s,v,y)A_{s}^{2}\,ds\\
& \leq  2(\mu+\sigma)\int_{0}^{t}\sum_{v\in\mathbb{Z}^{d}}p(t-s,x,v)A_{t}p^{2}(s,v,y)A_{s}\,ds\\
& \leq  2(\mu+\sigma)A^{2}\int_{0}^{t}\sum_{v\in\mathbb{Z}^{d}}p(t-s,x,v)p^{2}(s,v,y)ds\\
& \leq  2(\mu+\sigma)A^{2}\int_{0}^{t}p(s,0,0)\sum_{v\in\mathbb{Z}^{d}}p(t-s,x,v)p(s,v,y)ds\\
& \leq  2(\mu+\sigma)A^{2}p(t,x,y)\int_{0}^{t}p(s,0,0)ds\\
& \leq  2(\mu+\sigma)A^{2}p(t,x,y)G_{0}(0,0)\\
& \leq  2BA^{2}p(t,x,y).
\end{split}
\]
To prove this chain of inequalities we use the following facts.

\begin{enumerate}
\item $A_{t}\leq A$ and $A_{t-s}A_{s}\leq A$ because
\[
A_{t}=E_{0}\left[e^{\sigma\int_{0}^{t}\delta_{0}(X_{s})ds}\right]
\]
and $A=E_{0}\left[e^{\sigma\int_{0}^{\infty}\delta_{0}(X_{s})ds}\right]$,
\item $p(s,v,y)\leq p(s,0,0)$,
\item $\sum_{v\in\mathbb{Z}^{d}}p(t-s,x,v)p(s,v,y)=p(t,x,y)$.
\end{enumerate}
We will prove that
\[
m_{l-1}(t,x,y)\leq A^{l}B^{l-1}D_{l}p(t,x,y),
\]
where the sequence $D_{l}$ is recurrently defined as
\[
\begin{split}  D_{1}&=1,\\
  D_{l}&=\sum_{i=1}^{l-1}{l-1 \choose i}D_{i}D_{l-i}.
\end{split}
\label{D_l}
\]
Note that (\ref{D_l}) defines the sequence of the Catalan numbers, thus,
the exponential generating function for this sequence $D(z):=\sum_{l=1}^{\infty}D_{l}\frac{z^{l}}{l!}$
from (\ref{D_l}) satisfies $D^{2}(z)+z=D(z)$ or $D(z)=\frac{1-\sqrt{1-4z}}{2}$.
The last equality means that the $l$-th coefficient of $D(z)$ grows no faster
than $4^{l}$ or, equivalently, $D_{l}<4^{l}l!$.

For $l\ge2$, the $l$-th factorial moment satisfies:
\begin{align*}
\frac{\partial m_{l}}{\partial t} & = \mathcal{L}_{a}m_{l}+\sigma\delta_{0}(x)m_{l}+2(\mu+\sigma\delta_{0}(x))\sum_{i=1}^{l-1}\left(\begin{array}{c}
l-1\\
i
\end{array}\right)m_{i}m_{l-i},\\
m_{l}(0,x,y) & =  0.
\end{align*}
After applying Duhamel's formula, we have
\[
\begin{split}m_{l}(t,x,y)& = 2(\mu+\sigma\delta_{0}(x))\int_{0}^{t}\sum_{v\in\mathbb{Z}^{d}}m_{1}(t-s,x,v)\\
&\qquad\times\sum_{i=1}^{l-1}{l-1 \choose i}m_{i}(s,v,y)m_{l-i}(s,v,y)\,ds\\
& \leq 2(\mu+\sigma)D_{l}\int_{0}^{t}\sum_{v\in\mathbb{Z}^{d}}p(t-s,x,v)A_{t-s}A^{i}B^{i-1}\\
&\qquad\times p(s,v,y)A^{l-i}B^{l-i-1}p(s,v,y)ds\\
& \leq 2(\mu+\sigma)A^{l}B^{l-2}D_{l}\int_{0}^{t}\sum_{v\in\mathbb{Z}^{d}}p(t-s,x,v)A_{t-s}p^{2}(s,v,y)\,ds\\
& \leq 2(\mu+\sigma)A^{l+1}B^{l-2}D_{l}\int_{0}^{t}\sum_{v\in\mathbb{Z}^{d}}p(t-s,x,v)p^{2}(s,v,y)ds\\
& \leq 2(\mu+\sigma)A^{l+1}B^{l-2}D_{l}\int_{0}^{t}p(s,0,0)\sum_{v\in\mathbb{Z}^{d}}p(t-s,x,v)p(s,v,y)ds\\
&\leq 2(\mu+\sigma)A^{l+1}B^{l-2}D_{l}p(t,x,y)\int_{0}^{t}p(s,0,0)ds\\
&\leq A^{l+1}B^{l-1}D_{l}p(t,x,y).
\end{split}
\]
\end{proof}

\begin{theorem}\label{T4}
\label{main result} Let $N(t,y)$, $y\in\mathbb{Z}^{d}$, be a random
field as described above, and consider the single-source perturbation
case, that is, $\beta-\mu=\sigma\delta_{0}(x)$ with $\sigma >0$. If d$\ge3$ and $\sigma< G^{-1}_0(0,0)$,
then for all $y\in\mathbb{Z}^{d}$ we have
\[
N(t,y)\xrightarrow{\text{Law}}N(\infty,y).
\]
\end{theorem}
The proof of Theorem~\ref{T4} is based on Lemma~\ref{L4} and is completely identical to the proof of the main result from~\cite{SM_JW_2017}, which also uses the inequality $p(t,x,y)\leq p(t,0,0)$.

\section{The case of multiple-source perturbations}

Now, let us consider a more general case of perturbations. Assume that
each particle during the time interval $(t,t+dt)$ can
die with probability $\mu(x)\,dt$ or split, at the same point $x\in \mathbb{Z}^{d}$, into two particles with
probability $\beta(x)\,dt$. We assume
that $\mu(x)\equiv\mu$ and $\beta(x)=\mu+\sum_{i=1}^{k}\sigma_{i}\delta_{x_{i}}(x)$ with $\sigma_{i} > 0$. In this case
the first moment satisfies the following equation
\begin{align*}
\frac{\partial m_{1}}{\partial t} & = \mathcal{L}_{a}m_{1}+\sum_{i=1}^{k}\sigma_{i}\delta_{x_{i}}(x)m_{1},\\
m_{1}(0,x,y) & = \delta_{x}(y).
\end{align*}
In a similar way,  for the $l$-th factorial moment
\[
m_{l}(t,x,y):=E\left(n(t,x,y)\left(n(t,x,y)-1\right)\cdots\left(n(t,x,y)-l+1\right)\right), \quad l\geq2,
\]
we get the equation
\begin{align*}
\frac{\partial m_{l}}{\partial t} & = \mathcal{L}_{a}m_{l}+\sum_{i=1}^{k}\sigma_{i}\delta_{x_{i}}(x)m_{l}+2\beta(x)\sum_{i=1}^{l-1}\left(\begin{array}{c}
l-1\\
i
\end{array}\right)m_{i}m_{l-i},\\
m_{l}(0,x,y) & = 0.
\end{align*}

To prove the following lemma we use the definition
\[
C:=E_{0}\left[e^{\sum_{i=1}^{k}\sigma_{i}\int_{0}^{\infty}\delta_{0}(X_{s})ds}\right].
\]
\begin{lemma}
The following inequalities are valid: $m_{1}(t,x,0)\leq C\,p(t,x,0)$ and $m_{1}(t,x,y)\leq C\,p(t,x,y)$.
\end{lemma}

\begin{proof}
From the Kac-Feyman formula, we have the solution for $m_{1}(t,x,y)$:
\begin{align*}
m_{1}(t,x,0) &=  p(t,x,0)E_{x}\left[e^{\sum_{i=1}^{k}\sigma_{i}\int_{0}^{t}\delta_{x_{i}}(X_{s})ds}|X_{t}=0\right]\\
 & = p(t,x,0)E_{x}\left[e^{\sigma_{1}\int_{0}^{t}\delta_{x_{1}}(X_{s})ds}\cdots e^{\sigma_{k}\int_{0}^{t}\delta_{x_{k}}(X_{s})ds}|X_{t}=0\right].
\end{align*}
By the H\"{o}lder inequality we get
\[
E\left[X_{1}X_{2}\cdots X_{k}\right] \leq  \left(E|X_{1}|^{p_{1}}\right)^{\frac{1}{p_{1}}}\cdot\left(E|X_{2}|^{p_{2}}\right)^{\frac{1}{p_{2}}}\cdots\left(E|X_{k}|^{p_{k}}\right)^{\frac{1}{p_{n}}},
\]
where
\[
\frac{1}{p_{1}}+\frac{1}{p_{2}}+\cdots+\frac{1}{p_{k}}=1.
\]
If
\[
\frac{1}{p_{i}}=\frac{\sigma_{i}}{\sigma_{1}+\sigma_{2}+\cdots+\sigma_{k}},\qquad i=1,2,\ldots,k,
\]
and
\[
\sum_{i=0}^{k}\sigma_{i}<G_{0}^{-1}(0,0)
\]
then we have
\begin{align*}
m_{1}(t,x,0) & \le p(t,x,0)E_{x}\left[e^{\sigma_{1}p_{1}\int_{0}^{\infty}\delta_{x_{i}}(X_{s})ds}|X_{t}=0\right]^{\frac{1}{p_{1}}}\\
&\qquad\cdots E_{x}\left[e^{\sigma_{k}p_{k}\int_{0}^{\infty}\delta_{x_{i}}(X_{s})ds}|X_{t}=0\right]^{\frac{1}{p_{k}}}\\
 & \leq  p(t,0,0)E_{0}\left[e^{\sigma_{1}p_{1}\int_{0}^{\infty}\delta_{0}(X_{s})ds}|X_{t}=0\right]^{\frac{1}{p_{1}}}\\
 &\qquad\cdots E_{0}\left[e^{\sigma_{k}p_{k}\int_{0}^{\infty}\delta_{0}(X_{s})ds}|X_{t}=0\right]^{\frac{1}{p_{k}}}\\
 & = p(t,0,0)E_{0}\left[e^{\sum_{i=1}^{k}\sigma_{i}\int_{0}^{\infty}\delta_{0}(X_{s})ds}|X_{t}=0\right].
\end{align*}
In a similar way, one can prove that
\[
m_{1}(t,x,y)\leq p(t,x,y)C.
\]
\end{proof}
\begin{theorem}
For $d\ge3$,
\[
m_{l}(t,x,y)\leq C^{l-1}B^{l}l!p(t,x,y)
\]
where $C:=E_{0}\left[e^{\sum_{i=1}^{k}\sigma_{i}\int_{0}^{\infty}\delta_{0}(X_{s})ds}\right]$,
$C_{t}:=E_{0}\left[e^{\sum_{i=1}^{k}\sigma_{i}\int_{0}^{t}\delta_{0}(X_{s})ds}\right]$
and $B=2(\mu+\sum_{i=1}^{k}\sigma_{i})G_0(0,0)$.
\end{theorem}

\begin{proof}
The basic idea  is similar to that of the proof of Theorem~\ref{thm: single_source_moment}. We
will use the mathematical induction. For $l=2$ the second moment satisfies the following equation
\begin{align*}
\frac{\partial m_{2}}{\partial t} & = \mathcal{L}_{a}m_{2}+\sum_{i=1}^{k}\sigma_{i}\delta_{x_{i}}(x)m_{2}+2(\mu+\sum_{i=1}^{k}\sigma_{i}\delta_{x_{i}}(x))m_{1}^{2},\\
m_{2}(0,x,y) & = 0.
\end{align*}
From Duhamel's formula, we have
\[
\begin{split}m_{2}(t,x,y)&= 2(\mu+\sum_{i=1}^{k}\sigma_{i}\delta_{x_{i}}(x))\int_{0}^{t}\sum_{v\in\mathbb{Z}^{d}}m_{1}(t-s,x,v)m_{1}^{2}(s,v,y)\,ds\\
& \leq  2(\mu+\sum_{i=1}^{k}\sigma_{i})\int\sum_{v\in\mathbb{Z}^{d}}p(t-s,x,v)C_{t-s}p^{2}(s,v,y)C_{s}^{2}\,ds\\
& \leq  2(\mu+\sum_{i=1}^{k}\sigma_{i})\int_{0}^{t}\sum_{v\in\mathbb{Z}^{d}}p(t-s,x,v)C_{t}p^{2}(s,v,y)C_{s}\,ds\\
& \leq  2(\mu+\sum_{i=1}^{k}\sigma_{i}) C\int_{0}^{t}\sum_{v\in\mathbb{Z}^{d}}p(t-s,x,v)p^{2}(s,v,y)ds\\
& \leq  2(\mu+\sum_{i=1}^{k}\sigma_{i}) C^{2}\int_{0}^{t}p(s,0,0)\sum_{v\in\mathbb{Z}^{d}}p(t-s,x,v)p(s,v,y)ds\\
& \leq  2(\mu+\sum_{i=1}^{k}\sigma_{i}) C^{2}p(t,x,y)\int_{0}^{t}p(s,0,0)ds\\
& \leq  2(\mu+\sum_{i=1}^{k}\sigma_{i}) C^{2}p(t,x,y)G_{0}(0,0)\\
& \leq  2BC^{2}p(t,x,y)
\end{split}
\]
The rest of the proof is similar to the proof of Theorem~\ref{thm: single_source_moment}, so
we omit it here.
\end{proof}
\begin{theorem}
\label{main result-1} Let $N(t,y)$, $y\in\mathbb{Z}^{d}$, be a
random field as described above, and consider the multiple-source
perturbations case, i.e. $\beta-\mu=\sum_{i=1}^{k}\sigma_{i}\delta_{x_{i}}(x)$, for all $i=1,\ldots,k$ and $\sigma_{i} > 0$.
If $d\ge3$ and $\sum_{i=1}^{k}\sigma_{i}<G^{-1}_0(0,0)$  then
for all $y\in\mathbb{Z}^{d}$ we have
\[
N(t,y)\xrightarrow{\text{Law}}N(\infty,y).
\]
\end{theorem}
The proof of the  theorem is based on the scheme of the proof of Theorem~\ref{T4}, see details in~\cite{SM_JW_2017}.

\textbf{Acknowledgements.}
S.~Molchanov was partly supported by the Russian Science Foundation (RSF), project No. 17-11-01098.  E.~Yarovaya was supported by the Russian Foundation for Basic Research (RFBR), project No. 17-01-00468.


\begin{thebibliography}{99}
\bibitem{1998}S. Albeverio, L. V. Bogachev, and E. B. Yarovaya,
 Asymptotics of branching symmetric random walk on
the lattice with a single source, C. R. Acad.
Sci. Paris, S\'{e}r. 1: Math, 326(8), 975-980 (1998).

\bibitem{Bogachev_1998a}L. V. Bogachev, and E. B. Yarovaya,
A limit theorem for a supercritical branching random walk on $\mathbb{Z}^{d}$
with a single source. Uspehi Matematicheskih Nauk, 53, 229- 230 (1998).

\bibitem{Bogachev_1998b}L. V. Bogachev, and E. B. Yarovaya,
Moment analysis of a branching random walk on a lattice with a single
source. Doklady Akademii Nauk, 363, 439-442 (1998).

\bibitem{Feng_2012}Y. Feng, Molchanov,
S. and Whitmeyer J. Random walks with heavy tails and limit
theorems for branching processes with migration and immigration. Stochastic
and Dynamics, 12, 1-23 (2012).

\bibitem{Kondratiev_2008}Y. Kondratiev, O. Kutoviy, and S. Pirogov,
Correlation functions and invariant measures in continuous
contact model. Infin. Dimens. Anal. Quantum Probab. Relat. Top., 11(2),
231-258 (2008).

\bibitem{Kondratiev_2016}Y. Kondratiev, S. Molchanov, S. Pirogov,
E. Zhizhina, On ground state of non-local Schr\"{o}dinger
operators, Applicable Analysis, DOI: 10.1080/00036811.2016.1192138 (2016).

\bibitem{Kondratiev_Molchanov_2016}Y. Kondratiev, S. Molchanov, B.
Vainberg,  Spectral analysis of non-local Schr\"{o}dinger operators,
arXiv:1603.01626 (2016).



\bibitem{Koralov_2013}L. Koralov and S.
Molchanov, The structure of the population inside the propagating
front, Journal of Mathematical Sciences (Problems In Mathematical
Analysis) 189, 637-658 (2013).

\bibitem{SM_JW_2017}S. Molchanov and J.
Whitmeyer, Stationary distributions in Kolmogorov-Petrovski-
Piskunov-type models with an infinite number of particles Mathematical
Population Studies, Routledge, 24, 147-160 (2017).

\bibitem{Molchanov_Yarovaya_2012}S. Molchanov and E. Yarovaya,
\textquotedblleft Limit theorems for the Green function of the lattice
Laplacian under large deviations of the random walk,\textquotedblright{}
Izv. Ross. Akad. Nauk, Ser. Mat. 76(6), 123-152 (2012).

\bibitem{Molchanov_2013}S. Molchanov and E. Yarovaya, Large
deviations for a symmetric branching random walk on a multidimensional
lattice, Proceedings of the Steklov Institute of Mathematics, 282,
186-201 (2013).

\bibitem{Molchanov_Yarovaya_2012_a}S. Molchanov and E. Yarovaya,
Limit theorems for the Green function of the lattice Laplacian under
large deviations of the random walk. Izv. RAN. Ser. Mat., 76(6), 1190-1217(2012).

\bibitem{Molchanov_Yarovaya_2012_b}S. Molchanov and E. Yarovaya,
 Branching Processes with Lattice Spatial Dynamics and a Finite
Set of Particle Generation Centers. Doklady Mathematics, 86(2), 638-641 (2012).

\bibitem{Molchanov_Yarovaya_2012_c}S. Molchanov and E. Yarovaya, Population Structure inside the Propagation Front of a Branching
Random Walk with Finitely Many Centers of Particle Generation. Doklady
Mathematics, 86(2), 787-790 (2012).

\bibitem{Yarovaya_2007} E. Yarovaya, Branching random walks in a heterogeneous environment. Center of Applied Investigations of the Faculty of Mechanics and Mathematics of the Moscow State University, Moscow, in Russian (2007).

\end{thebibliography}
\end{document}